\documentclass[smallextended,envcountsect]{svjour3} 
\smartqed 
\usepackage{graphicx}
\journalname{JOTA}

\usepackage{amssymb}
\usepackage{amsmath}
\usepackage{algorithm}
\usepackage[noend]{algpseudocode}
\usepackage{longfbox}
\usepackage{hyperref}
\fboxset{padding-top=-10pt,padding-bottom=0pt}%
\floatstyle{plaintop}
\restylefloat{algorithm}
\def\LFT{[}
\def\RGH{]}

\def\ba{\begin{array}}
	\def\ea{\end{array}}
\def\la{\langle}
\def\ra{\rangle}

\def\E{\mathbb{E}}
\def\argmin{\mathop{\rm argmin}}
\def\Argmin{\mathop{\rm Argmin}}
\def\H{\mathcal{H}}
\def\Hf{\mathcal{H}_{\!f}}
\def\Hfnu{\mathcal{H}_{\!f}(\nu)}
\def\sigmaf{\sigma_{\!f}}

\def\sigmafp{\sigma_{\!f}(p)}

\def\sigmafnu{\sigma_{\!f}(2 + \nu)}

\def\gammafnu{\gamma_{f}(\nu)}
\def\S{\mathcal{S}}

\newcommand{\refLE}[1]{\ensuremath{\stackrel{(\ref{#1})}{\leq}}}
\newcommand{\refEQ}[1]{\ensuremath{\stackrel{(\ref{#1})}{=}}}

\newcommand{\defeq}{:=} 
\newcommand{\dom}{\operatorname{dom}} 

\def\BT{\begin{theorem}}
	\def\ET{\end{theorem}}
\def\BL{\begin{lemma}}
	\def\EL{\end{lemma}}
\def\BC{\begin{corollary}}
	\def\EC{\end{corollary}}
\def\BE{\begin{example}}
	\def\EE{\end{example}}
\def\BD{\begin{definition}}
	\def\ED{\end{definition}}
\def\BR{\begin{remark}}
	\def\ER{\end{remark}}
\def\BAS{\begin{assumption}}
	\def\EAS{\end{assumption}}
\def\BI{\begin{itemize}}
	\def\EI{\end{itemize}}

\newcommand{\half}{\mbox{${1 \over 2}$}}

\def\beq{\begin{equation}}
\def\eeq{\end{equation}}

\begin{document}

\title{Minimizing Uniformly Convex Functions by \\
	   Cubic Regularization of Newton Method\thanks{Communicated by Lionel Thibault.}}


\author{Nikita Doikov  \and  Yurii Nesterov }

\institute{Nikita Doikov,  Corresponding author  \at
             ICTEAM (Catholic University of Louvain), Louvain-la-Neuve, Belgium\\
              Nikita.Doikov@uclouvain.be. 
              ORCID: 0000-0003-1141-1625.
           \and
           Yurii Nesterov    \at
              CORE (Catholic University of Louvain),
              Louvain-la-Neuve, Belgium\\
              Yurii.Nesterov@uclouvain.be. 
              ORCID: 0000-0002-0542-8757.
}

\titlerunning{Minimizing Uniformly Convex Functions by Cubic Newton}

\date{Received: 16 July 2019 / Accepted: 15 February 2021
/ Published online: 10 March 2021
\;\textcopyright\,\!\!  The Author(s) 2021}

\maketitle

\begin{abstract}
In this paper, we study the iteration complexity of cubic
regularization of Newton method for solving composite
minimization problems with uniformly convex objective. We
introduce the notion of second-order condition number of a
certain degree and justify the linear rate of convergence
in a nondegenerate case for the method with an adaptive
estimate of the regularization parameter. The algorithm
automatically achieves the best possible global complexity
bound among different problem classes of uniformly convex
objective functions with H\"older continuous Hessian of
the smooth part of the objective. As a byproduct of our
developments, we justify an intuitively plausible result
that the global iteration complexity of the Newton method
is always better than that of the gradient method on the
class of strongly convex functions with uniformly bounded
second derivative.
\end{abstract}


\section{Introduction} \label{SectionIntroduction}
A big step in a second-order optimization theory is
related to the global complexity guarantees which were
justified in \cite{nesterov2006cubic} for the cubic
regularization of the Newton method. The following results
provide a good perspective for the development of this
approach, discovering
accelerated~\cite{nesterov2008accelerating},
adaptive~\cite{cartis2011adaptive1,cartis2011adaptive2}
and universal~\cite{grapiglia2017regularized} schemes. The
latter methods can automatically adjust to a smoothness
properties of the particular objective function. In the
same vein, the second-order algorithms for solving a
system of nonlinear equations were discovered
in~\cite{nesterov2007modified}, and randomized variants
for solving large-scale optimization problems were
proposed in~
\cite{cartis2018global,doikov2018randomized,ghadimi2017second,kohler2017sub,tripuraneni2018stochastic}.

Despite to a number of nice properties, global complexity
bounds of the cubically regularized Newton method for the
cases of strongly convex and uniformly convex objective
are not still fully investigated, as well as the notion of
second-order non-degeneracy (see discussion in Sect.
5 in \cite{nesterov2008accelerating}). We are going to
address this issue in the current paper.

The rest of the paper is organized as follows.
Sect. \ref{SectionUniformConvexity} contains all
necessary definitions and main properties of the classes
of uniformly convex functions and twice-differentiable
functions with H\"older continuous Hessian. We introduce
the notion of the \textit{condition number} $\gammafnu$ of
a certain degree $\nu \in [0, 1]$ and present some basic
examples.

In Sect. \ref{SectionRegularizedNewton}, we describe a general
regularized Newton scheme and show the linear rate of convergence for
this method on the class of uniformly convex functions with a known
degree $\nu \in [0, 1]$ of nondegeneracy. Then, we introduce the
adaptive cubically regularized Newton method and collect useful
inequalities and properties, which are related to this algorithm.

In Sect. \ref{SectionHolderHessianComplexity}, we study global
iteration complexity of the cubically regularized Newton method on
the classes of uniformly convex functions with H\"older continuous
Hessian. We show that for nondegeneracy of \textit{any} degree
$\nu \in [0, 1]$, which is formalized by the
condition $\gammafnu > 0$, the algorithm automatically achieves the
linear rate of convergence with the value $\gammafnu$ being the main
complexity factor.

Finally, in Sect. \ref{SectionDiscussion} we
compare our complexity bounds with the known bounds for
other methods and discuss the results. In particular, we
justify an intuitively plausible (but quite a delayed)
result that the global complexity of the
cubically regularized Newton method is
always better than that of the gradient method on the
class of strongly convex functions with uniformly bounded
second derivative.

\section{Uniformly Convex Functions with H\"older Continuous Hessian}
\label{SectionUniformConvexity}

Let us start from some notation. In what follows, we denote
by $\E$ a finite-dimensional real vector space and by
$\E^{*}$ its dual space, which is a space of linear
functions on $\E$. The value of function $s \in \E^{*}$ at
point $x \in \E$ is denoted by $\langle s, x \rangle$.
Let us fix some linear self-adjoint positive-definite operator
$B: \E \to \E^{*}$ and introduce the following Euclidean norms in the
primal and dual spaces:
$$
\ba{rcl}
\|x\| & \defeq &  \langle Bx, x \rangle^{1/2}, \quad x \in \E,
\qquad \|s\|_{*} \; \defeq \; \langle s, B^{-1}s \rangle^{1/2},
\quad s \in \E^{*}.
\ea
$$
For any linear operator $A: \E \to \E^{*}$, its norm is induced in a
standard way:
$$
\ba{rcl}
\|A\| & \defeq & \max\limits_{x \in \E} \bigl\{ \|Ax\|_{*} \; | \;
\|x\| \leq 1 \bigr\}.
\ea
$$
Our goal is to solve the convex optimization problem
in the composite form:
\beq \label{MainProblem}
\ba{rcl}
\min\limits_{x \in \dom F} F(x) & := & f(x) + h(x),
\ea
\eeq
where $f$ is a twice differentiable on its open domain
uniformly convex function, and $h$ is a \textit{simple}
closed convex function with $\dom h \subseteq \dom f$.
\textit{Simple} means that all auxiliary subproblems with
an explicit presence of $h$ are easily solvable.

For a smooth function $f$, its gradient at point $x$ is
denoted by $\nabla f(x) \in \E^{*}$,  and its Hessian is
denoted by $\nabla^2 f(x) : \E \to \E^{*}$. For convex but
not necessary differentiable function $h$, we denote by
$\partial h(x) \subset \E^{*}$ its subdifferential at the
point $x \in \dom h$.

We say that differentiable function $f$ is \textit{uniformly convex}
of degree $p \geq 2$ on a convex set $C \subseteq \dom f$
if for some constant $\sigma > 0$ it satisfies inequality
\beq \label{UniformlyConvexDefinition}
\ba{rcl}
f(y) & \; \geq \; & f(x) + \langle \nabla f(x), y - x \rangle +
\frac{\sigma}{p}\|y - x\|^p, \qquad x, y \in C.
\ea
\eeq
Uniformly convex functions of degree $p = 2$ are known as
\textit{strongly convex}. If
inequality~\eqref{UniformlyConvexDefinition} holds with
$\sigma = 0$, the function $f$ is called just \textit{convex}.
The following convenient condition
is sufficient for function $f$ to be uniformly convex on a convex set
$C \subseteq \dom f$:
\begin{lemma}(Lemma 1 in~\cite{nesterov2008accelerating})
	Let for some $\sigma > 0$ and $p \geq 2$
	the following inequality holds:
	\beq \label{UniformlyConvexSufficient}
	\ba{rcl}
	\langle \nabla f(x) - \nabla f(y), x - y \rangle & \; \geq \; &
	\sigma \|x - y\|^p, \qquad x, y \in C.
	\ea
	\eeq
	Then, function $f$ is uniformly convex of degree $p$ on set $C$ with
	parameter $\sigma$.
\end{lemma}
From now on, we assume
$
C \; := \; \dom F \; \subseteq \; \dom f.
$
By the composite representation~\eqref{MainProblem},
we have for every $x \in \dom F$ and for
all $F'(x) \in \partial F(x)$:
\begin{equation} \label{UniformlyConvexLowerBound}
\ba{rcl}
F(y) & \geq  & F(x) + \langle
F'(x), y - x \rangle + \frac{\sigma}{p}\|x - y\|^p, \qquad
y \in \dom F.
\ea
\end{equation}
Therefore, if $\sigma > 0$, then we can have only one
point $x^{*} \in \dom F$ with $F(x^{*}) = F^{*}$, which
always exists for $F$ being uniformly convex and closed.
A useful consequence of uniform convexity is the following
upper bound for the residual.
\begin{lemma} \label{UniformlyConvexConsequence}
	Let $f$ be uniformly convex of degree $p \geq 2$ with
	constant $\sigma > 0$ on set $\dom F$. Then, for every $x
	\in \dom F$ and for all $F'(x) \in \partial F(x)$ we have
	\begin{equation} \label{UniformlyConvexGradientBound}
	\ba{rcl}
	F(x) - F^{*} & \; \leq \; & \frac{p - 1}{p} \left( \frac{1}{\sigma}
	\right)^{\frac{1}{p - 1}} \|F'(x)\|_{*}^{\frac{p}{p - 1}}.
	\ea
	\end{equation}
\end{lemma}
\begin{proof}
	\noindent
	In view of~\eqref{UniformlyConvexLowerBound}, 
	bound \eqref{UniformlyConvexGradientBound}
	follows as in the proof of Lemma~3 
	in~\cite{nesterov2008accelerating}.
\qed
\end{proof}

It is reasonable to define the best possible constant~$\sigma$ in
inequality~\eqref{UniformlyConvexSufficient} for a certain degree~$p$.
This leads us to a system of constants:
\begin{equation} \label{SigmaDefinition}
\ba{rcl}
\sigmafp & \; \defeq \; & \inf\limits_{\substack{x, y \, \in \, \dom F \\
		x \not= y }} \frac{\langle \nabla f(x) - \nabla f(y), x - y \rangle }
{\|x - y\|^p}, \qquad p \geq 2.
\ea
\end{equation}
We prefer to use
inequality~\eqref{UniformlyConvexSufficient} for the
definition of $\sigmafp$, instead
of~\eqref{UniformlyConvexDefinition}, because of its
symmetry in $x$ and $y$. Note that the value $\sigmafp$
also depends on the domain of $F$. However, we omit this
dependence in our notation since it is always clear from
the context.

It is easy to see that the
univariate function $\sigma_{\!f}( \cdot )$ is log-concave.
Thus, for all $p_2 > p_1 \geq 2$ we have:
\begin{equation} \label{SigmaLogConcave}
\ba{rcl}
\sigma_{\!f}(p) & \; \geq \; &
\bigl( \sigma_{\!f}(p_1) \bigr)^{\frac{p_2 - p}{p_2 - p_1}} \cdot
\bigl( \sigma_{\!f}(p_2) \bigr)^{\frac{p - p_1}{p_2 - p_1}},
\qquad p \in [p_1, p_2].
\ea
\end{equation}

For a twice-differentiable function $f$, we say
that it has \textit{H\"older continuous Hessian} of degree
$\nu \in [0, 1]$ on a convex set $C \subseteq \dom f$, if for some
constant $\mathcal{H}$, it
holds:
\begin{equation} \label{HolderHessianDefinition}
\ba{rcl}
\| \nabla^2 f(x) - \nabla^2 f(y) \| & \leq &
\mathcal{H} \|x - y\|^{\nu}, \qquad x, y \in C.
\ea
\end{equation}
Two simple consequences of~\eqref{HolderHessianDefinition}
are as follows:
\beq \label{HolderHessianGradBound}
\ba{rcl}
\| \nabla f(y) - \nabla f(x) - \nabla^2 f(x)(y - x) \|_{*} 
& \leq &
 \frac{\H \|x - y\|^{1 + \nu}}{1 + \nu},
\ea
\eeq
\beq \label{HolderHessianFuncBound}
\ba{rcl}
| f(y) - Q(x; y) |  & \leq & 
\frac{\H \|x - y\|^{2 + \nu}}{(1 + \nu)(2 + \nu)},
\ea
\eeq
where $Q(x; y)$ is the quadratic model of $f$ at the point
$x$:
$$
\ba{rcl}
Q(x; y) & \; \defeq \; & f(x) + \langle \nabla f(x), y - x \rangle +
\frac{1}{2} \langle \nabla^2 f(x) (y - x), y - x \rangle.
\ea
$$

In order to characterize the level of smoothness of
function $f$ on the set $C := \dom F$, let us define
the system of H\"older constants (see
\cite{grapiglia2017regularized}):
\begin{equation} \label{HfnuDefinition}
\ba{rcl}
\Hfnu & \; \defeq &\; \sup\limits_{\substack{x, y \in \dom F \\ x \not= y}}
\frac{\| \nabla^2 f(x) - \nabla^2 f(y) \|}{\|x - y\|^{\nu}},
\qquad \nu \in [0, 1].
\ea
\end{equation}
We allow $\Hfnu$ to be equal to $+\infty$ for some $\nu$.
Note that function $\H_{\!f}( \cdot )$ is log-convex.
Thus, any $0 \leq \nu_1 < \nu_2 \leq 1$ such that
$\H_{\!f}(\nu_i) < +\infty, i = 1,2$, provide us with the
following upper bounds for the whole interval:
\begin{equation} \label{HLogConvex}
\ba{rcl}
\H_{\!f}(\nu) & \; \leq \; &
\bigl( \H_{\!f}(\nu_1) \bigr)^{\frac{\nu_2 - \nu}{\nu_2 - \nu_1}} \cdot
\bigl( \H_{\!f}(\nu_2) \bigr)^{\frac{\nu - \nu_1}{\nu_2 - \nu_1}},
\qquad \nu \in [\nu_1, \nu_2].
\ea
\end{equation}
If for some specific $\nu \in [0, 1]$ we have $\Hfnu = 0$,
this implies that $\nabla^2 f(x) = \nabla^2 f(y)$ for all
$x, y \in \dom F$. In this case restriction
$\left.f\right|_{\dom F}$ is a quadratic function and we
conclude that $\Hfnu = 0$ for \textit{all} $\nu \in [0,
1]$. At the same time, having two points $x, y \in \dom F$
with $0 < \|x - y\| \leq 1$, we get a simple uniform lower
bound for all constants $\Hfnu$:
$$
\ba{rcl}
\Hfnu & \; \geq \; & \| \nabla^2 f(x) - \nabla^2 f(y) \|,
\qquad \nu \in [0, 1].
\ea
$$

Let us give an example of function, which has H\"older
continuous Hessian for all $\nu \in [0, 1]$.
\begin{example}
	For a given  $a_i \in \E^{*}$, $1 \leq i \leq m$, consider the
	following convex function:
	$$
	\ba{rcl}
	f(x) & \; = \; & \ln \left( \sum_{i = 1}^m e^{\langle a_i,  x
		\rangle} \right), \quad x \in \E.
	\ea
	$$
	Let us fix Euclidean norm $\|x\| = \langle Bx, x
	\rangle^{1/2}, x \in \E$, with operator $B := \sum_{i
		= 1}^m a_i a_i^{*}$. Without loss of generality, we assume
	that $B \succ 0$ (otherwise we can reduce dimension of the
	problem). Then,
	$$
	\ba{rcl}
	\Hf(0) & \; \leq \; & 1, \quad  \Hf(1) \;\, \leq \,\; 2.
	\ea
	$$
	Therefore, by~\eqref{HLogConvex} we get, for any $\nu \in [0, 1]$:
	$$
	\ba{rcl}
	\Hfnu & \; \leq \;& 2^{\nu}.
	\ea
	$$
\end{example}
\begin{proof}
	Denote $\kappa(x) \equiv \sum_{i = 1}^m e^{\langle a_i, x \rangle}$.
	Let us fix arbitrary $x, y \in \E$ and direction $h \in \E$. Then,
	straightforward computation gives:
	$$
	\ba{rcl}
	\langle \nabla f(x), h \rangle &\; = \;& \frac{1}{\kappa(x)}
	\sum_{i = 1}^m e^{\langle a_i, x \rangle} \langle a_i, h \rangle, \\[6pt]
	\langle \nabla^2 f(x)h, h \rangle &\; =\;&
	 \frac{1}{\kappa(x)} \sum_{i = 1}^m e^{\langle a_i, x \rangle}
	\langle
	a_i, h \rangle^2 - \bigl( \frac{1}{\kappa(x)} \sum_{i = 1}^m
	e^{\langle a_i, x \rangle} \langle a_i, h \rangle  \bigr)^2 \\[5pt]
	&\; = \;& \frac{1}{\kappa(x)} \sum_{i = 1}^m e^{\langle a_i, x \rangle}
	\left( \langle
	a_i, h \rangle - \langle \nabla f(x), h \rangle
	\right)^2 \; \geq \; 0.
	\ea
	$$
	Hence, we get
	$$
	\ba{rcl}
	\| \nabla^2 f(x) \| & = & \max\limits_{\|h\| \leq 1} \langle \nabla^2 f(x)
	h, h \rangle \;
	\leq \;  \max\limits_{\|h\| \leq 1} \sum_{i = 1}^m \langle a_i, h \rangle^2
	\;
	= \; \max\limits_{\|h\| \leq 1} \|h\|^2 \; = \; 1.
	\ea
	$$
	Since all Hessians of function $f$ are positive definite,
	we conclude that $\Hf(0) \leq 1$. Inequality $\Hf(1) \leq
	2$ can be easily obtained from the following
	representation of the third derivative:
	$$
	\ba{rcl}
	f'''(x)[h,h,h]  & \; = \; & {1 \over \kappa(x)}
	\sum\limits_{i=1}^m e^{\la a_i, x \ra} \left( \la a_i, h
	\ra - \la \nabla f(x), h \ra \right)^3\\
	\\
	& \;  \leq \; & \la \nabla^2 f(x) h, h \ra \max\limits_{1 \leq
		i,j \leq m } \la a_i - a_j , h \ra \; \leq \; 2 \| h \|^3.
	\ea
	$$
\qed
\end{proof}

Let us imagine now that we want to describe the iteration
complexity of some method, which solves the composite
optimization problem~\eqref{MainProblem} up to an absolute
accuracy $\epsilon > 0$ in the function value. We assume
that the smooth part $f$ of its objective is uniformly
convex and has H\"older continuous Hessians. Which degrees
$p$ and $\nu$ should be used in our analysis? Suppose
that, for the number of \textit{calls of the oracle}, we
are interested in obtaining  a polynomial-time bound of
the form:
$$
\ba{c}
O\left((\Hfnu)^{\alpha} \cdot (\sigmafp)^{\beta} \cdot
\log \frac{F(x_0) - F^{*} \,}{\varepsilon}\right), \quad
\alpha,\beta \neq 0.
\ea
$$
Denote by $\LFT x \RGH$ the \textit{physical dimension} of
variable $x \in\E$, and by $\LFT f \RGH$ the
\textit{physical dimension} of the value $f(x)$. Then, we
have $\LFT \nabla f(x) \RGH = \LFT f \RGH / \LFT x \RGH$
and $\LFT \nabla^2f(x) \RGH = \LFT f \RGH / \LFT x
\RGH^2$. This gives us
\begin{displaymath}
\ba{c}
\LFT \Hfnu \RGH \; = \; \frac{ \LFT f \RGH }{ \LFT x \RGH^{2 + \nu} },
\quad \LFT \sigmafp \RGH = \frac{\LFT f \RGH}{\LFT x \RGH^p},
\quad \LFT \, (\Hfnu)^{\alpha} \cdot (\sigmafp)^{\beta} \, \RGH =
\frac{\LFT f \RGH^{\alpha + \beta}}{ \LFT x \RGH^{\alpha(2 + \nu) + \beta p} }.
\ea
\end{displaymath}
While $x$ and $f(x)$ can be measured in arbitrary physical
quantities, the value "number of iterations" {\em cannot
	have} physical dimension. This leads to the following
relations:
\begin{displaymath}
\alpha + \beta = 0 \qquad \text{and} \qquad \alpha (2 + \nu) + \beta p = 0.
\end{displaymath}
Therefore, despite to the fact that our function can belong to several
problem classes simultaneously, from the physical point of view only
one option is available:
\begin{displaymath}
\boxed{p = 2 + \nu}
\end{displaymath}

Hence, for a twice-differentiable convex function $f$ with
$\inf_{\nu \in [0, 1]} \Hfnu > 0$, we can define only one
meaningful \textit{condition number} of degree $\nu \in
[0, 1]$:
\begin{equation} \label{ConditionNumberDefinition}
\ba{rcl}
\gamma_f(\nu) &\; \defeq \;& \frac{\sigma_f(2 + \nu)}{\Hfnu}.
\ea
\end{equation}
If for some particular $\nu$ we have $\Hfnu = +\infty$, then by our
definition: $\gamma_f(\nu) = 0$.

It will be shown that the condition number $\gamma_f(\nu)$
serves as a main factor in the global iteration complexity
bounds for the regularized Newton method as applied to the
problem \eqref{MainProblem}. Let us prove that this number
cannot be big.

\begin{lemma}
	Let $\inf_{\nu \in [0, 1]} \Hfnu > 0$ and therefore the condition
	number $\gamma_f(\cdot)$ be well defined. Then,
	\begin{equation} \label{GammaNuUpperBound}
	\ba{rcl}
	\gamma_f(\nu) & \quad \leq \quad & \frac{1}{1 + \nu} \;\; + \;
	\inf\limits_{x, y \in \dom F} \frac{\| \nabla^2 f(x) \|}{\| \nabla ^2 f(y) -
		\nabla^2 f(x) \|}, \qquad \nu \in [0, 1].
	\ea
	\end{equation}
	In the case when $\dom F$ is unbounded: $\sup_{x \in \dom
		F} \|x\| = +\infty$, then,
	\begin{equation} \label{GammaNuUpperBound2}
	\ba{rcl}
	\gamma_f(\nu) & \quad \leq \quad & \frac{1}{1 + \nu}, \qquad \nu \in (0, 1].
	\ea
	\end{equation}
\end{lemma}
\begin{proof}
	Indeed, for any $x, y \in \dom F$, $x \not= y$, we have:
	$$
	\ba{rcl}
	\sigma_f(2 + \nu) \quad &\overset{\eqref{SigmaDefinition}}{\leq}& \quad
	\frac{\la \nabla f(y) - \nabla f(x), y - x \ra }{\|y - x\|^{2 + \nu}} \\[8pt]
	\quad
	& = & \quad \frac{\la \nabla f(y) -
		\nabla f(x) - \nabla^2 f(x)(y - x), y - x \ra  }{\|y - x\|^{2 + \nu}}
	\; + \;
	\frac{\la \nabla^2 f(x)(y - x), y - x \ra}{\|y - x\|^{2 + \nu}} \\[8pt]
	\quad
	& \overset{\eqref{HolderHessianGradBound}}{\leq} & \quad
	\frac{\Hfnu}{1 + \nu} \; + \; \frac{\|\nabla^2 f(x) \|}{\|y - x\|^{\nu}}.
	\ea
	$$
	Now, dividing both sides of this inequality by $\Hfnu$,
	we get inequality \eqref{GammaNuUpperBound} from the
	definition of $\Hfnu$~\eqref{HfnuDefinition}. Inequality
	\eqref{GammaNuUpperBound2} can be obtained by taking the
	limit $\|y\| \to +\infty$.
\qed
\end{proof}

From inequalities~\eqref{SigmaLogConcave}
and~\eqref{HLogConvex}, we can get the following lower
bound:
$$
\ba{rcl}
\gamma_f(\nu) & \; \geq \; & \bigl(  \gamma_f(\nu_1)
\bigr)^{\frac{\nu_2 - \nu}{\nu_2 - \nu_1}} \cdot \bigl(
\gamma_f(\nu_2) \bigr)^{\frac{\nu - \nu_1}{\nu_2 -
		\nu_1}}, \qquad \nu \in [\nu_1, \nu_2],
\ea
$$
where $0 \leq \nu_1 < \nu_2 \leq 1$. However, it turns out
that in \textit{unbounded case} we can have a nonzero
condition number $\gammafnu$ only for a {\em single
	degree}.
\begin{lemma}~\label{LemmaOneNonzero}
	Let $\dom F$ be unbounded: $\sup_{x \in \dom F} \|x\| = +\infty$.
	Assume that
	for a fixed $\nu \in [0, 1]$ we have $\gamma_f(\nu) > 0$. Then,
	$$
	\gamma_f(\alpha) = 0 \quad \text{for all} \quad \alpha \in [0, 1]
	\setminus \{ \nu \}.
	$$
\end{lemma}
\begin{proof}
	Consider firstly the case: $\alpha > \nu$. From the
	condition $\gamma_f(\nu) > 0$, we conclude that $\Hf(\nu) <
	+\infty$. Then, for any $x, y \in \dom F$ we have:
	$$
	\ba{rcl}
	\frac{\sigmaf(2 + \alpha) \|y - x\|^{2 + \alpha}}{2 + \alpha} \quad
	& \overset{\eqref{UniformlyConvexDefinition}}{\leq} &  \quad f(y) -
	f(x) - \langle \nabla f(x), y - x \rangle \\[8pt]
	& \overset{\eqref{HolderHessianFuncBound}}{\leq} & \quad
	\frac{1}{2}\langle \nabla^2 f(x)(y - x), (y - x) \rangle +
	\frac{\Hf(\nu) \|y - x\|^{2 + \nu}}{(1 + \nu)(2 + \nu)}.
	\ea
	$$
	Dividing both sides of this inequality by $\|y - x\|^{2 +
		\alpha}$ and letting $\|x\| \to +\infty$, we get
	$\sigmaf(2 + \nu) = 0$. Therefore, $\gamma_f(\alpha) = 0$.
	For the second case, $\alpha < \nu$, we cannot have
	$\gamma_f(\alpha) > 0$, since the previous reasoning
	results in $\gamma_f(\nu) = 0$.
\qed
\end{proof}

Let us look now at an important example of a uniformly
convex function with H\"older continuous Hessian. It is
convenient to start with some properties of powers of
Euclidean norm.
\BL\label{lm-ENormP}
For fixed real $p\geq 1$, consider the following function:
$$
\ba{rcl}
f_p(x) & \; = \; & \frac{1}{p} \|x \|^{p}, \quad x \in \E.
\ea
$$
1. For $p \geq 2$, function $f_p(\cdot)$ is uniformly
convex of degree $\mbox{\rm $p$:}$\footnote{$^{\!\!\!\!)}$ For the
	integer values of $p$, this inequality was proved in
	\cite{nesterov2008accelerating}.}$^{\!)}$
\begin{equation} \label{PoweredNormUniformConvexity}
\langle \nabla f_p(x) - \nabla f_p(y), x - y \rangle \quad
\geq \quad 2^{2 - p} \|x - y\|^{p}, \quad x, y \in \E.
\end{equation}
2. If $1 \leq p \leq 2$, then function $f_p(\cdot)$ has
$\nu$-H\"older continuous gradient with $\nu = p-1$:
\beq\label{eq-HCond}
\| \nabla f_p(x) - \nabla f_p(y) \|_* \leq 2^{1-\nu} \| x
- y \|^{\nu}, \quad x, y \in \E.
\eeq
\EL
\begin{proof}
	Firstly, recall two useful inequalities, which are valid
	for all $a, b \geq 0$:
	\begin{align}
	|a^{\alpha} - b^{\alpha}| \; \leq \; |a - b|^{\alpha},
	\quad &\text{when} \quad 0 \leq \alpha \leq 1, \label{PoweredNormNuIneq} \\[5pt]
	|a^{\alpha} - b^{\alpha}| \; \geq \; |a - b|^{\alpha},
	\quad &\text{when} \quad \alpha \geq 1.  \label{PoweredNormPIneq}
	\end{align}
	
	Let us fix arbitrary $x, y \in \E$. The left-hand side of
	inequality (\ref{PoweredNormUniformConvexity}) equals
	$$
	\la \|x\|^{p - 2}Bx - \|y\|^{p - 2}By, x - y \ra \; = \;
	\|x\|^p + \|y\|^p - \la Bx, y \ra ( \|x\|^{p - 2} +
	\|y\|^{p - 2} ),
	$$
	and we need to verify that it is bigger than
	$
	2^{2 - p}\bigl[ \|x\|^2 + \|y\|^2 - 2 \la Bx, y \ra
	\bigr]^{\frac{p}{2}}.
	$
	The case $x = 0$ or $y = 0$ is trivial. Therefore, assume
	$x \not= 0$ and $y \not= 0$. Denoting $\tau :=
	\frac{\|y\|}{\|x\|}$, $r := \frac{\langle Bx,
		y\rangle}{\|x\| \cdot \|y\|}$, we have the
	following statement to prove:
	$$ 
	\ba{rcl}
	1 + \tau^p & \geq & r \tau(1 + \tau^{p - 2}) + 2^{2 - p}
	\bigl[  1 + \tau^2 - 2 r \tau  \bigr]^{\frac{p}{2}} ,
	\quad \tau > 0, \quad |r| \leq 1.
	\ea
	$$
	Since the function in the right-hand side is convex in
	$r$, we need to check only two marginal cases:
	\begin{enumerate}
		\item $r = 1 \, : \quad $
		$1 + \tau^{p} \; \geq \; \tau (1 + \tau^{p - 2}) + 2^{2 - p} |1 - \tau|^p$,
		which is equivalent to $(1 - \tau) (1 - \tau^{p - 1}) \geq 2^{2 - p}|1 - \tau|^p$.
		This is true by~\eqref{PoweredNormPIneq}.
		
		\item $r = -1\, : \quad $ $1 + \tau^{p} \; \geq \; -\tau(1 + \tau^{p - 2}) + 2^{2 - p}(1 + \tau)^p$,
		which is equivalent to $(1 + \tau^{p - 1}) \geq 2^{2 - p}(1 + \tau)^{p - 1} $.
		This is true in view of convexity of function $\tau^{p-1}$
		for $\tau \geq 0$.
	\end{enumerate}
	Thus, we have proved~\eqref{PoweredNormUniformConvexity}.
	Let us prove the second statement. Consider the function
	$\hat f_q(s) = {1 \over q} \| s \|^q_*$, $s \in \E^*$,
	with $q = {p \over p-1} \geq 2$. In view of our first
	statement, we have:
	\beq\label{eq-Dop1}
	\ba{rcl}
	\la s_1 - s_2, \nabla \hat f_q(s_1) - \nabla \hat f_q(s_2)
	\ra 
	& \geq & \left(\half\right)^{q-2} \| s_1 - s_2 \|_*^q,
	\quad s_1, s_2 \in \E^*.
	\ea
	\eeq
	For arbitrary $x_1, x_2 \in \E$, define $s_i = \nabla
	f_p(x_i) = {B x_i \over  \| x_i \|^{2-p}} $, $i = 1, 2$.
	Then $\| s_i \|_* = \| x_i \|^{p-1}$, and consequently,
	$$
	\ba{rcl}
	x_i & = & \| x_i \|^{2-p} B^{-1} s_i \; = \; \| s_i \|_{*}^{2-p \over
		p-1} B^{-1} s_i \; = \; \nabla \hat f_q(s_i).
	\ea
	$$
	Therefore, substituting these vectors in (\ref{eq-Dop1}),
	we get
	$$
	\left(\half\right)^{q-2} \| \nabla
	f_p(x_1) - \nabla f_p(x_2) \|_*^q \leq \la \nabla f_p(x_1)
	- \nabla f_p(x_2), x_1 - x_2 \ra.
	$$
	Thus, $\| \nabla f_p(x_1) - \nabla f_p(x_2) \|_* \leq
	2^{q-2 \over q-1} \| x_1 - x_2 \|^{1 \over q-1}$. It
	remains to note that ${1 \over q-1} = p-1 = \nu$.
\qed
\end{proof}

\begin{example}
	For real $p \geq 2$ and arbitrary $x_0 \in \E$, consider
	the following function:
	$$
	\ba{rcl}
	f(x) & \; = \; & \frac{1}{p} \|x - x_0\|^{p} \; = \; f_p(x -
	x_0),
	\quad x \in \E.
	\ea
	$$
	Then, $\sigmaf(p) \; = \; \left(\frac{1}{2} \right)^{p -
		2}$. Moreover, if $p = 2 + \nu$ for some $\nu \in (0, 1]$,
	then it holds
	\begin{displaymath}
	\ba{rcl}
	\Hfnu & \; \leq \; & (1 + \nu)2^{1 - \nu},
	\ea
	\end{displaymath}
	and $\Hf(\alpha) \; = \; +\infty$, for all
	$\alpha \in [0, 1] \setminus \{\nu\}$. Therefore, in this case we have
	$
	\gammafnu \; \geq \; \frac{1}{2(1 + \nu)},
	$
	and $\gamma_{f}(\alpha) = 0$ for all
	$\alpha \in [0, 1] \setminus \{\nu\}$.
\end{example}
\begin{proof}
	Let us take an arbitrary $x \neq 0$ and set $y := -x$. Then,
	\begin{displaymath}
	\langle \nabla f(x) - \nabla f(y), y - x \rangle \; = \;
	\langle \|x\|^{p - 2} Bx + \|x\|^{p - 2} Bx, 2 x \rangle
	\; = \; 4 \|x\|^{p}.
	\end{displaymath}
	On the other hand, $\| y - x \|^p = 2^p \| x \|^p$.
	Therefore, $\sigmaf(p) \refLE{SigmaDefinition} 2^{2-p}$,
	and (\ref{PoweredNormUniformConvexity}) tells us that this
	inequality is satisfied as equality.
	
		Let us prove now that $\Hfnu \leq (1 + \nu)2^{1 - \nu}$
	for $p = 2 + \nu$ with some $\nu \in (0, 1]$. This is
	\begin{equation} \label{PoweredNormHolderHessian}
	\| \nabla^2 f(x) - \nabla^2 f(y) \| \; \leq \; (1 + \nu) 2^{1 - \nu}
	\|x - y\|^{\nu}, \quad x, y \in \E.
	\end{equation}
	The corresponding Hessians can be represented as follows:
	\begin{displaymath}
	\ba{rcl}
	\nabla^2 f(x) & = & \|x\|^{\nu} B + \frac{\nu B x x^{*} B}{\|x\|^{2 - \nu}},
	\quad x \in \E \setminus \{0\}, \qquad \nabla^2 f(0) = 0.
	\ea
	\end{displaymath}
	
	For the case $x = y = 0$, inequality
	(\ref{PoweredNormHolderHessian}) is trivial. Assume now
	that $x \not= 0$. If $0 \in [x, y]$, then $y = -\beta x$
	for some $\beta \geq 0$ and we have:
	$$
	\ba{rcl}
	\| \nabla^2 f(x) - \nabla^2 f(-\beta x) \| \;
	&\leq& \; |1 - \beta^\nu | (1 + \nu) \|x\|^{\nu} \; \leq \; 
	(1 + \beta)^{\nu} (1 + \nu) 2^{1 - \nu} \|x\|^{\nu} \\
	\\
	&=& \; (1 + \nu) 2^{1 - \nu} \|x - y\|^{\nu},
	\ea
	$$
	which is~\eqref{PoweredNormHolderHessian}. Let $0 \notin
	[x, y]$. For an arbitrary fixed direction $h \in \E$, we
	get:
	\begin{displaymath}
	\ba{rcl}
	\bigl|  \bigl\langle (\nabla^2 f(x) - \nabla^2 f(y)) h, h \bigr\rangle
	\bigr| & = & \Bigl| \left( \|x\|^{\nu} - \|y\|^{\nu} \right) \cdot
	\|h\|^2  + \nu \cdot
	\left( \frac{\langle Bx, h \rangle^2}{\|x\|^{2 - \nu}} -
	\frac{\langle By, h \rangle^2}{\|y\|^{2 - \nu}}  \right) \Bigr|.
	\ea
	\end{displaymath}
	Consider the points $u = \frac{Bx}{\|x\|^{1 - \nu}} =
	\nabla f_q(x)$ and $v = \frac{By}{\|y\|^{1 - \nu}} =
	\nabla f_q(y)$ with $q = 1+\nu$. Then,
	\begin{displaymath}
	\ba{c}
	\|x\|^{\nu} = \|u\|_*, \quad  \frac{\langle B x, h
		\rangle^2}{\|x\|^{2 - \nu}} = \frac{\langle u, h
		\rangle^2}{\|u\|_*} \quad \text{and} \quad \|y\|^{\nu} =
	\|v\|_*, \quad \frac{\langle By, h \rangle^2}{\|y\|^{2 -
			\nu}} = \frac{\langle v, h \rangle^2}{\| v \|_*}.
	\ea
	\end{displaymath}
	Therefore,
	\begin{equation} \label{HessianNewVariables}
	\ba{cl}
	& \bigl|  \bigl\langle (\nabla^2 f(x) - \nabla^2 f(y)) h, h
	\bigr\rangle \bigr| \\[5pt]
	& \quad =  \Bigl| \left( \|u\|_* -
	\|v\|_* \right) \cdot
	\|h\|^2 \; + \; \nu
	\cdot \left( \frac{\langle u, h \rangle^2}{\|u\|_*}
	- \frac{\langle v, h \rangle^2}{\|v\|_*}
	\right) \Bigr|.
	\ea
	\end{equation}

	Let us estimate the right-hand side
	of~\eqref{HessianNewVariables} from above. Consider a
	continuously differentiable univariate function:
	\begin{displaymath}
	\ba{c}
	\phi(\tau) \; := \; \|u(\tau)\|_* \cdot \|h\|^2 + \nu
	\cdot \frac{\langle u(\tau), h \rangle^2}{\|u(\tau)\|_*},
	\quad u(\tau) \; := \; u + \tau (v -
	u), \quad \tau \in [0, 1].
	\ea
	\end{displaymath}
	Note that
	$$
	\ba{rcl}
	\phi^{\prime}(\tau) \; &=& \;  \frac{\langle u(\tau),
		B^{-1}(v-u)\rangle}{\|u(\tau)\|_*} \cdot \|h\|^2 + \frac{2
		\nu \langle u(\tau), h \rangle \langle v-u, h
		\rangle}{\|u(\tau)\|_*} \; - \; \frac{\nu\langle u(\tau),
		h \rangle^2 \langle u(\tau), B^{-1}(v-u)
		\rangle}{\|u(\tau)\|_*^3}
	\\[8pt]
	&=& \; \frac{\langle u(\tau), B^{-1}(v-u)
		\rangle}{\|u(\tau)\|_*} \cdot \underbrace{\left( \|h\|^2
		- \tfrac{\nu \langle u(\tau), h \rangle^2}{\|u(\tau)\|_*^2}
		\right)}_{\geq 0} \; + \; \frac{2\nu \langle u(\tau), h
		\rangle \langle v-u, h\rangle}{\| u(\tau) \|_*}.
	\ea
	$$
	Denote $\gamma := \frac{\langle u(\tau), h
		\rangle}{\|u(\tau)\|_* \cdot \|h\|} \in [-1, 1]$. Then,
	\begin{displaymath}
	\bigl| \phi^{\prime}(\tau) \bigr| \; \leq \; \|v - u\|_*
	\cdot \|h\|^2 \cdot \bigl(1 - \nu \gamma^2 + 2\nu|\gamma|
	\bigr) \; \leq \; (1 + \nu) \cdot \|v-u\|_* \cdot \|h\|^2.
	\end{displaymath}
	Thus, we have:
	\begin{equation} \label{PoweredNormSpecialBound}
	\bigl|  \bigl\langle (\nabla^2 f(x) - \nabla^2 f(y)) h, h
	\bigr\rangle \bigr| \; = \; | \phi(1) - \phi(0) | \; \leq
	\; (1 + \nu)  \cdot \|v - u\|_* \cdot \|h\|^2.
	\end{equation}
	It remains to use the definition of $u$ and $v$ and apply
	inequality (\ref{eq-HCond}) with $p=q$.
	Thus, we have proved, that for $p = 2 + \nu$ the Hessian of
	$f$ is H\"older continuous of degree~$\nu$. At the same
	time, taking $y = 0$, we get $\| \nabla^2 f(x) - \nabla^2
	f(y) \| = \| \nabla^2 f(x) \| = (1 + \nu)\|x\|^{\nu}$.
	These values cannot be uniformly bounded in $x \in \E$ by
	any multiple of $\|x\|^{\alpha}$ with $\alpha \neq \nu$.
	So, the Hessian of $f$ is \textit{not} H\"older continuous
	for any degree different from $2+\nu$.
\qed
\end{proof}

\BR
Inequalities (\ref{PoweredNormUniformConvexity}) and
(\ref{eq-HCond}) have the following symmetric
consequences:
$$
\ba{lll}
p \geq 2 & \Rightarrow & \| \nabla f_p(x) - \nabla f_p(y)
\|_* \; \geq \; 2^{2-p} \| x - y \|^{p-1}, \\
\\
p \leq 2 & \Rightarrow & \| \nabla f_p(x) - \nabla f_p(y)
\|_* \; \leq \;  2^{2-p} \| x - y \|^{p-1},
\ea
$$
which are valid for all $x, y \in \E$. 
\ER

\section{Regularized Newton Method} \label{SectionRegularizedNewton}

Let us start from the case when we know that for a specific
$\nu \in [0, 1]$ function $f$ has H\"older continuous Hessian:
$\H_f(\nu) < +\infty$. Then, from~\eqref{HolderHessianFuncBound}, we
have the global upper bound for the objective function:
\begin{displaymath}
\ba{rcl}
F(y) & \; \leq \; & M_{\nu, H}(x; y) \;\defeq \; Q(x; y)
+ \frac{H \|x - y\|^{2 + \nu}}{(1 + \nu)(2 + \nu)} + h(y),
\qquad x, y \in \dom F,
\ea
\end{displaymath}
where $H > 0$ is large enough: $H \geq \Hfnu$. Thus, it is
natural to employ the minimum of a regularized quadratic
model:
\begin{displaymath}
\ba{c}
T_{\nu, H}(x) \; \defeq \;
\argmin\limits_{y \in \dom F} M_{\nu, H}(x; y), \qquad M_{\nu, H}^{*}(x)
\; \defeq \; \min\limits_{y \in \dom F} M_{\nu, H}(x; y),
\ea
\end{displaymath}
and define the following general iteration
process~\cite{grapiglia2017regularized}:
\begin{equation} \label{GeneralRegularizedNewton}
\boxed{ \quad x_{k + 1} \; := \; T_{\nu, H_k}(x_k), \qquad k \geq 0 }
\end{equation}
where the value $H_k$ is chosen either to be a constant
from the interval $[0, 2\Hfnu]$ or by some adaptive
procedure.

For the class of uniformly convex functions of degree $p =
2 + \nu$, we can justify the following global convergence
result for this process.
\begin{theorem}
	Assume that for some $\nu \in [0, 1]$ we have $0 < \Hfnu <
	+\infty$ and $\sigmafnu > 0$. Let the coefficients $\{ H_k
	\}_{k \geq 0}$ in the
	process~\eqref{GeneralRegularizedNewton} satisfy the
	following conditions:
	\begin{equation} \label{TheoremFixedNuConditionH}
	0 \leq H_k \leq \beta\H_f(\nu), \qquad F(x_{k + 1}) \leq M_{\nu,
		H_k}^{*}(x_k), \qquad k \geq 0,
	\end{equation}
	with some constant $\beta \geq 0$. Then, for the sequence
	$\{x_k\}_{k \geq 0}$ generated by the process we have:
	\begin{equation} \label{TheoremFixedNuOneStepProgress}
	\ba{rcl}
	F(x_{k + 1}) - F^{*} & \leq &
	\Bigl( 1 \; - \; \frac{1 + \nu}{2 + \nu} \cdot
	\min\Bigl\{\frac{\gammafnu (1 + \nu)}{(1 + \beta) (2 + \nu)}, \, 1
	\Bigr\}^{\frac{1}{1 + \nu}} \Bigr) \left( F(x_k) - F^{*} \right).
	\ea
	\end{equation}
	Thus, the rate of convergence is linear and for reaching
	the gap $F(x_K) - F^{*} \leq \varepsilon$ it is enough to
	perform
	$
	K \; = \;  \bigl \lceil \frac{2 + \nu}{1 + \nu} \cdot
	\max\bigl\{
	\frac{(1 + \beta)(2 + \nu)}{\gammafnu (1 + \nu)}, \, 1
	\bigr\}^{\frac{1}{1 + \nu}}
	\log \frac{F(x_0) - F^{*}}{\varepsilon} \bigr \rceil
    $
	iterations.
\end{theorem}
\begin{proof}
	
	As in the proof of Theorem 3.1 in~\cite{grapiglia2017regularized}, from~\eqref{TheoremFixedNuConditionH} one can see that
	$$
	\ba{rcl}
	F(x_{k + 1}) & \leq &  F(x_k) - \alpha \left( F(x_k) -
	F^{*} \right) + \alpha^{2 + \nu} \frac{(1 + \beta) \Hfnu \|x_k -
		x^{*}\|^{2 + \nu}}{(1 + \nu)(2 + \nu)},
	\ea
	$$
	for any $\alpha \in [0, 1]$.
	Then, taking into account the uniform 
	convexity~\eqref{UniformlyConvexLowerBound}, we get
	$$
	\ba{rcl}
	F(x_{k + 1}) 
	& \leq & 
	F(x_k)
	- \left( \alpha - \alpha^{2 + \nu} \frac{(1 + \beta) \Hfnu}{(1 +
		\nu) \sigmafnu} \right) \left( F(x_k) - F^{*} \right).
	\ea
	$$
	The minimum of the right-hand side is attained at
	$\alpha^{*} = \min  \bigl\{ \frac{ \gammafnu (1 + \nu)}{(2 + \nu)(1 +
		\beta)}, 1
	\bigr\}^{\frac{1}{1 + \nu}}\!\!$. Plugging this value
	into the bound above, we get
	inequality~\eqref{TheoremFixedNuOneStepProgress}.
\qed
\end{proof}

Unfortunately, in practice it is difficult to decide on an
appropriate value of $\nu \in [0, 1]$ with $\Hfnu <
+\infty$. Therefore, it is interesting to develop the
\textit{universal methods} which are not based on some
particular parameters. Recently, it was
shown~\cite{grapiglia2017regularized} that one good
choice for such universal scheme is the cubic
regularization of the Newton Method
\cite{nesterov2006cubic}. This is actually the
process~\eqref{GeneralRegularizedNewton} with the fixed
parameter $\nu = 1$. For this choice, in the rest part of
the paper we omit the corresponding index in the
definitions of all necessary objects:
$M_H(x; y) := M_{1, H}(x; y)$, 
$T_H(x) := T_{1, H}(x)$, and $M_H^{*}(x) := M_{1, H}^{*}(x) = M_H(x; T_H(x))$.
The adaptive scheme of our method with dynamic estimation
of the constant $H$ is as follows.

\vspace*{5pt}
\noindent
\hspace*{5pt}
\lfbox{
	\hspace*{4pt}
	\centering
	\begin{minipage}[b]{0.85\linewidth} \label{MainAlgorithm}
		\begin{algorithm}[H] 
		  	\caption{\small \textbf{Adaptive Cubic
					Regularization of Newton Method}}
			\vspace*{6pt}
			\noindent\makebox[\linewidth]{\rule{1.097\linewidth}{0.4pt}}
			\begin{algorithmic}[1]
				\Require Choose $x_0 \in \dom F$, $H_0 > 0$.
				\vspace*{2pt}
				\Ensure $k \geq 0$. \vspace*{2pt}
				\State Find the minimal integer $i_k \geq 0 \;$
				such that $F(T_{H_k 2^{i_k}}(x_k)) \leq M^{*}_{H_k
					2^{i_k}} (x_k)$. 
				\State Perform the Cubic Step: $x_{k + 1}  = T_{H_k
					2^{i_k}}(x_k)$. \vspace*{1pt}
				\State Set $H_{k + 1} := 2^{i_k - 1} H_k $.
				\vspace*{5pt}
			\end{algorithmic}
		\end{algorithm}
	\end{minipage}
	\hspace*{4pt}
}
\vspace*{1pt}

Let us present the main properties of the composite Cubic
Newton step $x \mapsto T_H(x)$. Denote
$$
r_H(x) \defeq \|T_H(x) - x\|.
$$
Since point $T_H(x)$ is a minimum of strictly convex
function $M_H(x;\cdot)$, it satisfies the following
first-order optimality condition:
\beq \label{CubicStepStationary}
\ba{cl}
&\bigl \la \nabla f(x) + \nabla^2 f(x)(T_H(x) - x) +
\tfrac{H r_H(x)}{2}B(T_H(x) - x), y - T_H(x) \bigr \ra  \; + \\
\\
& \quad h(y) \; \geq \; h(T_H(x)),
\qquad y \in \dom F.
\ea
\eeq 
In other words, the vector
$$
\ba{rcl}
h'(T_H(x)) & \defeq & -\nabla f(x) - \nabla^2 f(x)(T_H(x) -
x) - \frac{H r_H(x)}{2}B(T_H(x) - x)
\ea
$$
belongs to the subdifferential of $h$:
\begin{equation} \label{CubicStepSubdiff}
\ba{rcl}
h'(T_H(x)) &\; \in \; & \partial h(T_H(x)).
\ea
\end{equation}
Computation of a point $T = T_H(x)$, satisfying condition
\eqref{CubicStepSubdiff}, requires some standard techniques
of Convex Optimization and Linear Algebra
(see~\cite{nesterov2006cubic,nesterov2019implementable,agarwal2017finding,carmon2016gradient}). Arithmetical
complexity of such a procedure is usually similar to that
of the standard Newton step.

Plugging into~\eqref{CubicStepStationary} $y := x \in
\dom F$, we get:
\begin{align} 
&\la \nabla f(x), x - T_H(x) \ra \label{StationaryConsequence} \\
&\geq \; \la \nabla^2 f(x) (T_H(x) - x), T_H(x) - x \ra +
\tfrac{H r_H^3(x)}{2} + h(T_H(x)) - h(x). \notag
\end{align}
Thus, we obtain the following bound for the minimal value
$M_H^{*}(x)$ of the cubic model:
\beq \label{CubicModelMinimumExpression}
\ba{rcl}
M_H^{*}(x) 
& \overset{\eqref{StationaryConsequence}}{\leq} &
f(x) -
\tfrac{1}{2}\la \nabla^2 f(x)(T_H(x) - x), T_H(x) - x \ra -
\tfrac{H r_H^3(x)}{3} + h(x) \\[5pt]
& = &
F(x) - \tfrac{1}{2}\la \nabla^2 f(x)(T_H(x) - x),
T_H(x) - x \ra - \tfrac{Hr_H^3(x)}{3}. 
\ea
\eeq

If for some value $\nu \in [0, 1]$ the Hessian is H\"older continuous:
$\Hfnu < +\infty$, then by \eqref{HolderHessianGradBound}
and~\eqref{CubicStepSubdiff} we get the following bound for
the subgradient:
$$
F'(T_H(x)) \defeq \nabla f(T_H(x)) + h'(T_H(x))
$$
at the new point:
\beq \label{CubicGradNewPoint}
\ba{cl}
& \|  F'(T_H(x)) \|_{*} \\[5pt]
& \leq \;
  \| \nabla f(T_H(x)) - \nabla f(x) - \nabla^2 f(x)
(T_H(x) - x)\|_{*} + \tfrac{H r^2_H(x)}{2} \\[5pt]
& \stackrel{\eqref{HolderHessianGradBound}}{\leq}
\tfrac{\Hfnu
	r^{1 + \nu}_H(x)}{1 + \nu}  +  \tfrac{H r^2_H(x)}{2} 
\; = \; 
r^{1 + \nu}_H(x) \cdot \Bigl(  \tfrac{\Hfnu}{1 + \nu} + 
\tfrac{Hr_H^{1 - \nu}(x)}{2} \Bigr). 
\ea
\eeq

One of the main strong point of the classical Newton's is
its local \textit{quadratic convergence} for the class of
strongly convex functions with Lipschitz continuous
Hessian: $\sigmaf(2) > 0$ and $0 < \Hf(1) < +\infty$ (see,
for example,~\cite{nesterov2018lectures}). This property
holds for the cubically regularized Newton as
well~\cite{nesterov2006cubic,nesterov2008accelerating}.
Indeed, ensuring $F(T_H(x)) \leq M_{H}^{*}(x)$ as in Algorithm~1, and
having $H \leq \beta \Hf(1)$ with some $\beta \geq 0$, we get:
$$
\ba{rcl}
F(T_H(x)) - F^{*}
& \overset{\eqref{UniformlyConvexGradientBound}}{\leq} &
\frac{1}{2\sigmaf(2)}\| F'(T_H(x)) \|_{*}^2 
\; \overset{\eqref{CubicGradNewPoint}}{\leq} \;
\frac{(1 + \beta)^2 \Hf^2(1)}{8 \sigmaf(2)} r_H^4(x) \\
\\
& \leq & 
\frac{(1 + \beta)^2 \Hf^2(1)}{8 \sigmaf^3(2)}
\la \nabla^2 f(x)(T_H(x) - x), T_H(x) - x \ra^2 \\
\\
& \overset{\eqref{CubicModelMinimumExpression}}{\leq} &
\frac{(1 + \beta)^2 \Hf^2(1)}{2 \sigmaf^3(2)}
\left( F(x)- F^{*} \right)^2.
\ea
$$
And the region of quadratic convergence is as follows:
\begin{displaymath}
\ba{rcl}
\mathcal{Q} & \; = \; & \bigl\{ \, x \in \dom F \; : \;
F(x) - F^{*} \; \leq \;  \frac{2\sigmaf^3(2)}{(1 + \beta)^2 \Hf^2(1)}
\, \bigr\}.
\ea
\end{displaymath}
After reaching it, the method starts to double the right
digits of the answer at every step, and this cannot last
for a long time. Therefore, from now on we are mainly
interested in the \textit{global complexity bounds} of
Algorithm~1, which work for an arbitrary starting
point $x_0$.

For noncomposite case, as it was shown
in~\cite{grapiglia2017regularized}, if for some $\nu \in
[0, 1]$ we have $0 < \Hfnu < +\infty$ and the objective is
just \textit{convex}, then Algorithm 1 with small
initial parameter $H_0$ generates a solution $\hat{x}$
with $f(\hat{x}) - f^{*} \leq \varepsilon$ in
$
O\bigl( \bigl(\frac{\Hfnu D_0^{2 +
		\nu}}{\varepsilon}\bigr)^{\frac{1}{1 + \nu}} \bigr)
$
iterations, where $D_0 \; := \;
\max\limits_{x}\left\{ \|x - x^{*}\| \; : \; f(x) \leq
f(x_0) \right\}$. Thus, the method
in~\cite{grapiglia2017regularized} has a sublinear rate of
convergence on the class of convex functions with H\"older
continuous Hessian. It can \textit{automatically adapt} to
the actual level of smoothness. In what follows we show
that the same algorithm achieves linear rate of
convergence for the class of \textit{uniformly convex}
functions of degree $p = 2 + \nu$, 
namely for functions with strictly positive
condition number:
$
\sup_{\nu \in [0, 1]} \gammafnu > 0.
$

In the remaining part of the paper, we usually assume that
the smooth part of our objective is not \textit{purely
	quadratic}. This is equivalent to the condition
$\inf_{\nu \in [0, 1]} \Hfnu > 0$. However, to
conclude this section, let us briefly discuss the case
$\min_{\nu \in [0, 1]} \Hfnu = 0$.
If we would know in advance that $f$ is a convex quadratic
function, then no regularization is needed since a single
step $x \mapsto T_H(x)$ with $H := 0$ solves the
problem. However, if our function is given by a black-box
oracle and we do not know a priori that its smooth part is
quadratic, then we can still use  Algorithm~1. For this
case, we prove the following simple result.

\begin{proposition}
	Let $A: \E \to \E^{*}$ be a self-adjoint positive semidefinite linear
	operator and $b \in \E^{*}$. Assume that
	$
	f(x) \; := \; \frac{1}{2}\langle Ax, x \rangle - \langle b, x \rangle,
	$
	and the minimum
	$
	x^{*} \in \Argmin_{x \in \dom F} \bigl\{ F(x)
	:= f(x) + h(x)\bigr\}
	$
	does exist. Then, in order to get $F(x_K) - F^{*} \leq
	\varepsilon$ with arbitrary $\varepsilon
	> 0$, it is enough to perform
	\begin{equation} \label{CubicNewtonQuadraticComplexity}
	\ba{rcl}
	K & \; = \; & \bigl\lceil \log_2 \frac{H_0 \|x_0 - x^{*}\|^3}{6
		\varepsilon} \; + \; 1 \bigr\rceil
	\ea
	\end{equation}
	iterations of Algorithm 1.
\end{proposition}
\begin{proof}
	In our case, the quadratic model coincides with the smooth
	part of the objective:
	$
	Q(x; y) \equiv  f(y), \; x, y \in \E.
	$
	Therefore, at every iteration $k \geq 0$ of Algorithm 1 we  have
	$i_k = 0$ and $H_k = 2^{-k} H_0$.
	Note that $ x_{k + 1} = T_{2^{-k} H_0}(x_k) =
	\argmin_{y \in \dom F}\bigl\{ F(y) +
	\tfrac{2^{-k}H_0}{6}\|y - x_k\|^3 \bigr\}$, and
	\begin{equation} \label{CubicNewtonQuadraticFunctionProgress}
	\ba{rcl}
	F(x_{k + 1}) & \; \leq \; & F(y) + \frac{2^{-k} H_0}{6}\|y -
	x_k\|^3, \qquad y \in \dom F.
	\ea
	\end{equation}
	Let us prove that $\|x_{k + 1} - x^{*}\| \leq \|x_k -
	x^{*}\|$ for all $k \geq 0$. If this is true, then
	plugging $y \equiv x^{*}$
	into~\eqref{CubicNewtonQuadraticFunctionProgress}, we get:
	$F(x_{k + 1}) - F^{*} \leq 2^{-k}\frac{H_0}{6}\|x_0 -
	x^{*}\|^3$ which results in the
	estimate~\eqref{CubicNewtonQuadraticComplexity}.
	Indeed,
	$$
	\ba{rcl}
	\|x_k - x^{*}\|^2  & = &  \|(x_k - x_{k + 1}) + (x_{k + 1} -
	x^{*}) \|^2  \\[5pt]
	& = & \|x_{k + 1} - x^{*}\|^2  + \| x_{k} - x_{k + 1} \|^2 
	+ 2\langle B(x_k - x_{k + 1}), x_{k + 1} - x^{*} \rangle,
	\ea
	$$
	and it is enough to show that $\langle B(x_k - x_{k + 1}),
	x^{*} - x_{k + 1} \rangle \; \leq \; 0$.
	 Since $x_{k + 1}$ satisfies the first-order optimality
	condition:
	\begin{equation} \label{QuadraticFunctionStationaryCondition}
	\ba{rcl}
	-2^{-(k + 1)}H_0 \|x_{k + 1} - x_k\| B(x_{k + 1} - x_k)
	& \; \defeq \; & F'(x_{k + 1}) \; \in \; \partial F(x_{k + 1}),
	\ea
	\end{equation}
	we have:
	\begin{displaymath}
	\ba{rcl}
	\langle B(x_k - x_{k + 1}), x^{*} - x_{k + 1} \rangle
	& \overset{\eqref{QuadraticFunctionStationaryCondition}}{=} &
	\frac{2^{k + 1}}{H_0 \|x_k - x_{k + 1}\|} \langle F'(x_{k + 1}),
	x^{*} - x_{k + 1} \rangle \; \leq \; 0,
	\ea
	\end{displaymath}
	where the last inequality follows from the convexity of
	the objective.
\qed
\end{proof}

\section{Complexity Results for Uniformly Convex Functions}
\label{SectionHolderHessianComplexity}

In this section, we are going to justify the global linear
rate of convergence of Algorithm 1 for a class of twice
differentiable uniformly convex functions with H\"older
continuous Hessian. Universality of this method is ensured
by the adaptive estimation of the parameter $H$ over the
whole sequence of iterations. It is important to
distinguish two cases: $H_{k + 1} < H_k$ and $H_{k + 1}
\geq H_{k}$.

First, we need to estimate the progress in the objective
function after minimizing the cubic
model. There are two different
situations here:
$$
\mbox{either $H r^{1 - \nu}_H(x) \leq \frac{2 \Hfnu}{1 +
		\nu}$, or $H r^{1 - \nu}_H(x) > \frac{2 \Hfnu}{1 + \nu}$.}
$$
\begin{lemma} \label{LemmaStepHolder}
	Let $0 < \Hfnu < +\infty$ and $\sigmafnu > 0$ for some
	$\nu \in [0,1]$. Then, for arbitrary $x \in \dom F$ and $H
	> 0$ we have:
	\beq \label{OneStepProgressHolder}
	\ba{cl}
	& F(x) - M_H^{*}(x)  \\[5pt] 
	& \quad \geq \; \min\Bigl[ \left( F(x) -
	F^{*} \right) \cdot \frac{(1 + \nu)}{(2 + \nu)} \cdot
	\min\bigl\{\bigl( \frac{(1 + \nu) \gammafnu}{2(2 + \nu)}  \bigr)^{\frac{1}{1 + \nu}}, 
	\; 1\bigr\} ,   \\[5pt]
	& \; \qquad \qquad \left( F(T_H(x)) - F^{*} \right)^{\frac{3(1 + \nu)}{2(2 + \nu)}}
	\cdot \bigl ( \frac{2 + \nu}{1 + \nu}  \bigr)^{\frac{3(1 + \nu)}{2(2
			+ \nu)}} \cdot \frac{(\sigmafnu)^{\frac{3}{2(2 + \nu)}}}{ 3 \sqrt{H}
	}  \;  \Bigr]. 
	\ea
	\eeq
\end{lemma}
\begin{proof}
	Let us consider two cases. 1) 
	$H r_H^{1 - \nu}(x) \leq \frac{2 \Hfnu}{1 + \nu}$.
		Then, for arbitrary $y \in \dom F$, we have:
		$$
		\ba{rcl}
		M_H^{*}(x) 
		& := &  Q(x; T_H(x)) +
		\frac{H}{6}\|T_H(x) - x\|^3 + h(T_H(x)) \\[7pt]
		& \leq &  Q(x; y) + \frac{H r_H^{1 - \nu}(x) \|y -
			x\|^{2 + \nu}}{2 (2 + \nu) } + h(y) \\[5pt]
		& \overset{\eqref{HolderHessianFuncBound}}{\leq} &  
		F(y) +
		\frac{\Hfnu \|y - x\|^{2 + \nu}}{(1 + \nu)(2 + \nu)} +
		\frac{H r_H^{1 - \nu}(x) \|y - x\|^{2 + \nu}}{2 (2 + \nu) }
		\\[7pt]
		& \leq &
		 F(y) +  \frac{2\Hfnu \|y - x\|^{2 +
		\nu}}{(1 + \nu)(2 + \nu)},
		\ea
		$$
		where the first inequality follows from the fact, that
		$$
		\ba{rcl}
		T_H(x) & = & \argmin\limits_{y \in \dom F} \bigl\{ Q(x; y)
		+ \frac{H r_H^{1 - \nu}(x) \|y -
			x\|^{2 + \nu}}{2 (2 + \nu) } + h(y) \bigr\}.
		\ea
		$$
		Let us restrict $y$ to the segment: $y = \alpha x^{*} + (1 -
		\alpha) x,$ with $\alpha \in [0, 1]$. Taking into account the
		uniform convexity, we get:
		$$
		\ba{rcl}
		M_H^{*}(x) & \leq &  F(x) - \alpha \left( F(x) -
		F^{*} \right) + \alpha^{2 + \nu} \frac{2\Hfnu \|x^{*} - x\|^{2 +
				\nu}} {(1 + \nu)(2 + \nu)} \quad \\[5pt]
		& \overset{\eqref{UniformlyConvexLowerBound}}{\leq} &  
		F(x)
		-  \Bigl(\alpha  - \alpha^{2 + \nu} \frac{2\Hfnu}{  (1 + \nu)
			\sigmafnu}
		\Bigr) \left( F(x) - F^{*} \right).
		\ea
		$$
		The minimum of the right-hand side is attained at
		$
		\alpha^{*} = \min\bigl\{\frac{(1 + \nu)\gammafnu}{2(2 + \nu)}, 1 \bigr\}^{\frac{1}{1 + \nu}}.
		$
		Plugging this value into the bound, we have:
		\begin{displaymath}
		\ba{rcl}
		M^{*}_H(x) & \leq & F(x)  -   \min\bigl\{ \bigl(
		\frac{(1 + \nu)\gammafnu}{2(2 + \nu)}   \bigr)^{1 / (1 + \nu)}
		, \; 1 \bigr\} \cdot \frac{(1 + \nu)}{(2 + \nu)} \cdot \left(
		F(x) - F^{*} \right),
		\ea
		\end{displaymath}
		and this is the first argument of the minimum in
		\eqref{OneStepProgressHolder}.
		
		2) $H r_H^{1 - \nu}(x) > \frac{2 \Hfnu}{1 + \nu}.$
		By~\eqref{CubicGradNewPoint}, we have the bound:
		\begin{equation} \label{NextPointGradBound2}
		\ba{rcl}
		\| F'(T_H(x)) \|_{*} & \; < \; & Hr_H^2(x).
		\ea
		\end{equation}
		Using the fact that $\nabla^2 f(x) \succeq 0$, we get the second argument of the
		minimum:
	$$
	\ba{rcl}
	F(x) - M_H^{*}(x)
	& \overset{\eqref{CubicModelMinimumExpression}}{\geq} &
	\frac{Hr^3_H(x)}{3} \quad  \overset{ \eqref{NextPointGradBound2}
	}{\geq} \quad \frac{\| F'(T_H(x)) \|_{*}^{\frac32}}{3 \sqrt{H}}
	\\[5pt]
	& \overset{\eqref{UniformlyConvexGradientBound}}{\geq} &
	 \left(
	\frac{2 + \nu}{1 + \nu}  \right)^{\frac{3(1 + \nu)}{2(2 + \nu)}}
	\cdot \frac{(\sigmafnu)^{\frac{3}{2(2 + \nu)}}}{ 3 \sqrt{H} } \cdot
	\left( F(T_H(x)) - F^{*} \right)^{\frac{3(1 + \nu)}{2(2 + \nu)}}.
	\ea
	$$
\qed
\end{proof}

Denote by $\kappa_f(\nu)$ the following auxiliary value:
\begin{equation} \label{KappaDefinition}
\ba{rcl}
\kappa_f(\nu)
& \; \defeq \; & \frac{  \Hfnu^{\frac{2}{1 + \nu}}   }
{ (\sigmafnu)^{\frac{1 - \nu}{(1 + \nu)(2 + \nu)}} }
\cdot
\frac{6 \cdot (8 + \nu)^{\frac{1 - \nu}{1 + \nu}} }
{ \left((1 + \nu)(2 + \nu)\right)^{\frac{2}{1 + \nu}} }
\cdot
\bigl( \frac{1 + \nu}{2 + \nu} \bigr)^{\frac{1 - \nu}{2 + \nu}}, \;
\nu \in [0, 1].
\ea
\end{equation}
The next lemma shows what happens when parameter $H$ is increasing
during the iterations.

\begin{lemma} \label{LemmaUpHolder} Assume that for a fixed
	$x \in \dom F$ the parameter $H > 0$ is such that:
	\begin{equation} \label{LemmaUpHolderCondition}
	\ba{rcl}
	F(T_H(x)) & \; > \; & M^{*}_H(x).
	\ea
	\end{equation}
	If for some $\nu \in [0, 1]$, we have $\sigmafnu > 0$, then it holds:
	\begin{equation} \label{LemmaUpHolder3}
	\ba{rcl}
	H \left( F(T_{2H}(x)) - F^{*} \right)^{\frac{1 - \nu}{2 + \nu}}
	& \; < \; & \kappa_f(\nu).
	\ea
	\end{equation}
\end{lemma}
\begin{proof}
	Firstly, let us prove that from~\eqref{LemmaUpHolderCondition} we
	have:
	\begin{equation} \label{LemmaUpHolderBound1}
	\ba{rcl}
	H r^{1 - \nu}_H(x) & \; < \; & \frac{6 \Hfnu}{(1 + \nu)(2 + \nu)}.
	\ea
	\end{equation}
	Assuming by contradiction, $H r^{1 - \nu}_H(x) \; \geq \; \frac{6
		\Hfnu}{(1 + \nu)(2 + \nu)}$, we get:
	$$
	\ba{rcl}
	M_H^{*}(x)  & := & \frac{H \|T_H(x) - x\|^3}{6} + Q(x;
	T_H(x)) + h(T_H(x)) \\[5pt]
	& \geq & \frac{\Hfnu \|T_H(x) - x\|^{2 + \nu}}{(1 + \nu)(2 +
		\nu)} + Q(x; T_H(x)) + h(T_H(x)) \\[5pt]
	& \overset{\eqref{HolderHessianFuncBound}}{\geq} & F(T_H(x)),
	\ea
	$$
	which contradicts~\eqref{LemmaUpHolderCondition}.
	Secondly, by its definition, $M^*_H(x)$ is a concave function of $H$.
	Therefore, its derivative
	$
	{d \over d H} M^*_H(x)  =  {1 \over 6} r_H^3(x)
	$
	is non-increasing. Hence, it holds:
	\beq \label{LemmaUpHolderBound2}
	\ba{rcl}
	r_{2H}(x) & \leq & r_H(x)
	\overset{\eqref{LemmaUpHolderBound1}}{<} 
	\bigl( \frac{6 \Hfnu}{(1 + \nu)(2 + \nu) H} \bigr)^{\frac{1}{1 - \nu}}.
	\ea
	\eeq
	Finally, by the smoothness and the uniform convexity, we obtain:
	$$
	\ba{cl}
	& H \left( F(T_{2H}(x)) - F^{*} \right)^{\frac{1 - \nu}{2 + \nu}}
	\;\; \refLE{UniformlyConvexGradientBound} \;\;
	H \left(
	\frac{1 + \nu}{2 + \nu} \bigl( \frac{1}{\sigmafnu}
	\bigr)^{\frac{1}{1 + \nu}}  \right)^{\frac{1 - \nu}{2 + \nu}} 
	\| F'(T_{2H}(x)) \|_{*}^{\frac{1 - \nu}{1 + \nu}}  \\[5pt]
	& \quad \refLE{CubicGradNewPoint} 
	H \left(
	\frac{1 + \nu}{2 + \nu} \bigl( \frac{1}{\sigmafnu}
	\bigr)^{\frac{1}{1 + \nu}}  \right)^{\frac{1 - \nu}{2 + \nu}} 
	\Bigl(
	r^{1 +\nu}_{2H}(x) \cdot \left(  \frac{\Hfnu}{1 + \nu} + H r^{1 -
		\nu}_{2H}(x) \right)
	\Bigr)^{\frac{1 - \nu}{1 + \nu}} \\[5pt]
	& \quad \overset{\eqref{LemmaUpHolderBound2}}{<} 
	H \left(
	\frac{1 + \nu}{2 + \nu} \bigl( \frac{1}{\sigmafnu}
	\bigr)^{\frac{1}{1 + \nu}}  \right)^{\frac{1 - \nu}{2 + \nu}} 
	\Bigl( 
	r^{1+\nu}_{2H}(x) \cdot \frac{(8 + \nu) \Hfnu}{(1 + \nu)(2 + \nu)}
	\Bigr)^{\frac{1 - \nu}{1 + \nu}} \\[5pt]
	& \quad \overset{\eqref{LemmaUpHolderBound2}}{<} 
		\left(
	\frac{1 + \nu}{2 + \nu} \bigl( \frac{1}{\sigmafnu}
	\bigr)^{\frac{1}{1 + \nu}}  \right)^{\frac{1 - \nu}{2 + \nu}} 
	\left( \frac{\Hfnu}{(1 + \nu)(2 + \nu)}  \right)^{\frac{2}{1 + \nu}}
	6  (8 + \nu)^{\frac{1 - \nu}{1 + \nu}}
	\; =: \; \kappa_f(\nu).
	\ea
	$$
\qed
\end{proof}

We are ready to prove the main result of this paper.
\begin{theorem} \label{TheoremLinearRateUniformly} Assume that for a
	fixed $\nu \in [0, 1]$ we have $0 < \Hfnu < +\infty$ and
	$\sigmafnu > 0$. Let parameter $H_0$ in Algorithm 1 be
	small enough:
	\begin{equation} \label{HolderChoosingH}
	\ba{rcl}
	H_0 & \; \leq \; & \frac{ \kappa_f(\nu) }{ \left( F(x_0) - F^{*}
		\right)^{(1 - \nu) / (2 + \nu)  } },
	\ea
	\end{equation}
	where $\kappa_f(\nu)$ is defined
	by~\eqref{KappaDefinition}. Let the sequence $\{ x_k \}_{k
		= 0}^K$ generated by the method satisfy condition:
	\begin{equation} \label{HolderConditionFuncEps}
	F(T_{H_k2^{j}}(x_k)) - F^{*} \; \geq \; \varepsilon \; > \; 0, \qquad
	\quad 0 \leq j \leq i_k, \quad 0 \leq k \leq K - 1.
	\end{equation}
	Then, for every $0 \leq k \leq K - 1$, we have:
	\beq \label{HolderMainResult}
	\ba{cl}
	& F(x_{k + 1}) - F^{*}  \leq \\[5pt]
	& \bigl( 1 - \min\bigl\{  \frac{ (2 + \nu)\left( (1 + \nu) (2
		+ \nu) \right)^{ 1 / (1 + \nu) }  
	    \left( \gammafnu  \right)^{\frac{1}{1 +
	    		\nu}}   }{ (1 + \nu)6^{3/2} \cdot 2^{1/2} \cdot (8
		+ \nu)^{ (1 - \nu) / (2 + 2 \nu) } } , \frac{1}{2} \bigr\}  \bigr) \cdot 
	  \left( F(x_{k}) - F^{*} \right).
	\ea
	\eeq
	Therefore, the rate of convergence is linear, and
	\begin{displaymath}
	\ba{rcl}
	K & \; \leq \; & \max\bigl\{  \left( \gammafnu \right)^{\frac{-1}{1
			+ \nu}} \cdot \frac{1 + \nu}{2 + \nu}  \cdot \frac{6^{3/2} \cdot
		2^{1/2} \cdot (8 + \nu)^{(1 - \nu) / (2 + 2\nu)} }{ \left( (1 +
		\nu)(2 + \nu) \right)^{1 / (1 + \nu)}  } , \; 1 \bigr\} \cdot \log
	\frac{ F(x_0) - F^{*} }{\varepsilon}.
	\ea
	\end{displaymath}
	Moreover, we have the following bound for the total number of oracle
	calls $N_K$ during the first $K$ iterations:
	\begin{equation} \label{HolderOracleCallsBound}
	\ba{rcl}
	N_K & \; \leq \; & 2K  +  \log_2
	\frac{\kappa_f(\nu)}{\varepsilon^{(1 - \nu) / (2 + \nu)}} -
	\log_2 H_0.
	\ea
	\end{equation}
\end{theorem}
\begin{proof}
	The proof is based on Lemmas~\ref{LemmaStepHolder}
	and~\ref{LemmaUpHolder}, and monotonicity of the sequence
	$\bigl\{ F(x_k)  \bigr\}_{k \geq 0}$.
		Firstly, we need to show that every iteration of the method is
	well-defined. Namely, we are going to verify that for a fixed $0 \leq
	k \leq K-1$, there exists a finite integer $\ell \geq 0$ such that
	either $F(T_{H_k 2^{\ell}}(x_k) ) \leq M_{H_k 2^{\ell}}^{*}(x_k)$ or
	$F(T_{H_k 2^{\ell + 1}}(x_k)) - F^{*} < \varepsilon$.
		Indeed, let us set
	\begin{equation} \label{Th2HLowerBound}
	\ba{rcl}
	\ell := \max\left\{0,  \log_2 \left \lceil
	\frac{\kappa_f(\nu)}{H_k \varepsilon^{(1 - \nu) / (2 + \nu)} }
	\right\rceil  \right\}, 
	& \quad \text{and} \quad &
	H := H_k
	2^{\ell}  \geq  \frac{\kappa_f(\nu)}{\varepsilon^{(1 - \nu) / (2
			+ \nu)}}.
	\ea
	\end{equation}
	Then, if we have both $F(T_{H}(x_k)) > M_{H}^{*}(x_k)$ and
	$F(T_{2H}(x_k)) - F^{*}
	\geq \varepsilon$, we get by Lemma \ref{LemmaUpHolder}:
	\begin{displaymath}
	\ba{rcl}
	H & \; \overset{\eqref{LemmaUpHolder3}}{<} \; &
	\frac{\kappa_f(\nu)}{\left( F(T_{2H}(x_k)) - F^{*}  \right)^{(1 - \nu)
			/ (2 + \nu)}  } \; \leq \; \frac{\kappa_f(\nu)}{\varepsilon^{(1
			- \nu) / (2 + \nu)}},
	\ea
	\end{displaymath}
	which contradicts~\eqref{Th2HLowerBound}. Therefore, if we are unable
	to find the value $0 \leq i_k \leq \ell$ (see line~1 of
	Algorithm) in a finite number of steps, that only means we have
	already solved the problem up to accuracy $\varepsilon$.

	Now, let us show that for every $0 \leq k \leq K$ it holds:
	\begin{equation} \label{Th2UsefulBound}
	\ba{rcl}
	H_k \left( F(x_k) - F^{*} \right)^{\frac{1 - \nu}{2 + \nu}} 
	& \leq &
	\max \left\{ \kappa_f(\nu), \; H_0 \left( F(x_0) - F^{*}
	\right)^{\frac{1 - \nu}{2 + \nu}} \right\}.
	\ea
	\end{equation}
	This inequality is obviously valid for $k = 0$. Assume it
	is also valid for some $k \geq 0$. Then, by definition of
	$H_{k + 1}$ (see line~3 of Algorithm), we have $H_{k +
		1} = H_k 2^{i_k - 1}$. 
	There are two cases.
		1) $i_k = 0$. Then, $H_{k + 1} < H_k$. By monotonicity of
		$\bigl\{ F(x_k) \bigr\}_{k \geq 0}$ and by induction, we get:
		$$
		\ba{rcl}
		H_{k + 1} \left( F(x_{k + 1}) - F^{*} \right)^{\frac{1 - \nu}{2 +
				\nu}}  & \; < \; & H_k \left( F(x_k) - F^{*} \right)^{\frac{1 -
				\nu}{2 + \nu}} \\[5pt]
		& \; \leq \; & 
		\max \left\{ \kappa_f(\nu), \; H_0 \left( F(x_0) -
		F^{*} \right)^{\frac{1 - \nu}{2 + \nu}} \right\}.
		\ea
		$$
		
		2) $i_k > 0$. Then, applying Lemma~\ref{LemmaUpHolder} with $H
		:= H_k 2^{i_k - 1} = H_{k + 1}$ and $x :=
		x_{k}$, we have:
		\begin{displaymath}
		\ba{rcl}
		H_{k + 1} \left( F(x_{k + 1}) - F^{*} \right)^{\frac{1 - \nu}{2 +
				\nu}} & \; = \; & H \left( F(T_{2H}(x)) - F^{*}
		\right)^{\frac{1 - \nu}{2 + \nu}} \;
		\overset{\eqref{LemmaUpHolder3}}{\leq} \; \kappa_f(\nu).
		\ea
		\end{displaymath}
	Thus, \eqref{Th2UsefulBound} is true by induction.
	Choosing $H_0$ small enough~\eqref{HolderChoosingH}, we have:
	\begin{equation} \label{HolderHtGoodBound}
	\ba{rcl}
	2H_{k} \left( F(x_k) - F^{*}\right)^{\frac{1 - \nu}{2 + \nu}} 
	& \; \leq \; & 
	2\kappa_f(\nu), \qquad 0 \leq k \leq K.
	\ea
	\end{equation}
	From Lemma~\ref{LemmaStepHolder} we know, that one of the two
	following estimates is true (denote $\delta_k := F(x_k) - F^{*}$):

	\begin{enumerate}
		\item[1)] $F(x_k) - F(x_{k + 1})  \geq \alpha \cdot \delta_k
		 \;  \Leftrightarrow  \;  \delta_{k + 1}  \leq (1 -
		\alpha) \cdot \delta_k, \; $ \textit{or}
		
		\item[2)] $F(x_k) - F(x_{k + 1})  \geq \beta \cdot \delta_{k +
			1}  \; \Leftrightarrow \;  \delta_{k + 1}
	 	\leq (1 + \beta)^{-1} \delta_k  \leq  (1 - \min\{\beta,
		1\} / 2) \cdot \delta_k$,
	\end{enumerate}
	where
	$
	\alpha  
	 := 
	 \frac{1 + \nu}{2 + \nu} \cdot
	\min\bigl\{\bigl(  \frac{(1 + \nu) \gammafnu}{2(2 + \nu)} \bigr)^{
		\frac{1}{1 + \nu}}, \; 1\bigr\},
	$
	and
	$$
	\ba{rcl}
	\beta 
	 & := & 
	\bigl( \frac{2 + \nu}{1 + \nu}
	\bigr)^{\frac{3(1 + \nu)}{2(2 + \nu)}} \cdot
	\frac{(\sigmafnu)^{\frac{3}{2(2 + \nu)}}}{3 (2 \kappa_f(\nu))^{1/2}
	}
	\; \refEQ{KappaDefinition} \;
	\frac{2 + \nu}{1 + \nu} \cdot \frac{ 2^{1/2} \cdot
		\left( (1 + \nu) (2 + \nu)  \right)^{\frac{1}{1 + \nu}} }{6^{3/2}
		\cdot (8 + \nu)^{(1 - \nu) / (2 + 2\nu)}} \cdot \gammafnu^{\frac{1}{1
			+ \nu}}.
	\ea
	$$
	It remains to notice, that $\alpha \geq \min \bigl\{\beta, 1\bigr\} /
	2$. Thus, we obtain~\eqref{HolderMainResult}.
	
	Finally, let us estimate the total number of the oracle calls $N_K$
	during the first $K$ iterations. At each
	iteration, the oracle is called $i_k + 1$ times, and we have $H_{k +
		1} = H_k 2^{i_k - 1}$. Therefore,
	$$
	\ba{rcl}
	N_K & \; = \; & \sum_{k = 0}^{K - 1} (i_k + 1) \; = \; \sum_{k = 0}^{K -
		1} \left( \log_2 \frac{H_{k + 1}}{H_k} + 2\right) \\[5pt]
	& \; = \; & 2K + \log_2
	H_{K} - \log_2 H_0
	\; \overset{\eqref{HolderHtGoodBound},
		\eqref{HolderConditionFuncEps}}{\leq} \;
	2K + \log_2
	\frac{\kappa_f(\nu)}{\varepsilon^{(1 - \nu) / (2 + \nu)}} - \log_2
	H_0.
	\ea
	$$
\qed
\end{proof}

Note that condition~\eqref{HolderChoosingH} for the
initial choice of $H_0$ can be seen as a definition of the
moment, after which we can guarantee the linear rate of
convergence~\eqref{HolderMainResult}. In practice, we can
launch Algorithm~1 with \textit{arbitrary} $H_0 > 0$.
There are two possible options: either the method halves
$H_k$ at every step in the beginning, so $H_k$ becomes
small very quickly, or this value is increased at least
once, and the required bound is guaranteed by
Lemma~\ref{LemmaUpHolder}. It can be easily proved, that
this initial phase requires no more than $ K_0 =
\bigl\lceil \log_2 \frac{H_0 \varepsilon^{(1 - \nu) / (1 +
		\nu)}}{\kappa_f(\nu)} \bigr\rceil $ oracle calls.

\section{Discussion} \label{SectionDiscussion}

Let us discuss the global complexity results, provided by
Theorem~\ref{TheoremLinearRateUniformly} for the Cubic
Regularization of the Newton method with the adaptive
adjustment of the regularization parameter.

For the class of twice continuously differentiable
strongly convex functions with Lipschitz continuous
gradients $f \in \S_{\mu, L}^{2, 1}(\dom F)$, it is well
known that the classical gradient descent method needs
\begin{equation} \label{OLMuComplexity}
\ba{c}
O\bigl( \frac{L}{\mu} \log \frac{F(x_0) -
	F^{*}}{\varepsilon} \bigr)
\ea
\end{equation}
iterations for computing $\varepsilon$-solution of the
problem (e.g., \cite{nesterov2018lectures}). As it was
shown in \cite{cartis2012evaluation}, this result is
shared by a variant of Cubic Regularization of the Newton
method. This is much better than the
bound 
$
O\bigl( \bigl(\frac{L}{\mu}\bigr)^2
\log\frac{F(x_0) - F^{*}}{\varepsilon}\bigr),
$
known
for the damped Newton method~(e.g.,  \cite{boyd2004convex}).

For the class of uniformly convex functions of degree $p = 2 + \nu$
having H\"older continuous Hessian of degree $\nu \in [0, 1]$, we have
proved the following parametric estimates:
$
O\bigl( \max\bigl\{ \bigl(\gammafnu\bigr)^{\frac{-1}{1 + \nu}},
1\bigr\} \cdot \log \frac{F(x_0) - F^{*}}{\varepsilon} \bigr),
$
where $\gammafnu \defeq
\frac{\sigmafnu}{\Hfnu}$ is the \textit{condition number}
of degree $\nu$. However, in practice we may not know
exactly an appropriate value of the parameter $\nu$. It is
important that our algorithm automatically adjusts to the
best possible complexity bound:
\begin{equation} \label{OGammaComplexity}
\ba{c}
O\bigl( \max\bigl\{ \; \inf_{\nu \in [0, 1]}
\bigl(\gammafnu\bigr)^{\frac{-1}{1 + \nu}}, \; 1 \;\bigr\}
\cdot \log \frac{F(x_0) - F^{*}}{\varepsilon} \bigr).
\ea
\end{equation}
Note that for $f \in S_{\mu, L}^{2, 1}(\dom F)$ we have:
\begin{displaymath}
\ba{rcl}
\| \nabla^2 f(x) - \nabla^2 f(y) \| & \; \leq \; & L -
\mu, \qquad x, y \in \dom F.
\ea
\end{displaymath}
Thus, $\Hf(0) \leq L-\mu$ and $\gamma_f(0) \geq \frac{\mu}{L - \mu}$.
So we can conclude that the estimate~\eqref{OGammaComplexity}
is better than~\eqref{OLMuComplexity}.
Moreover, addition to our objective \textit{arbitrary} convex
quadratic function does not change any of $\Hfnu, \, \nu \in [0, 1]$.
Thus it can only improve the condition number $\gammafnu$, while the
ratio $L / \mu$ may become arbitrarily bad. It confirms an intuition
that a natural Newton-type minimization scheme should not be affected
by any quadratic parts of the objective, and the notion of
\textit{well-conditioned} and \textit{ill-conditioned} problems
for second-order methods should be different from that of for
first-order ones.

Note that in the recent paper~\cite{grapiglia2019accelerated},
a linear rate of convergence was also proven for the accelerated
second-order scheme, with the complexity bound:
\beq \label{OAccBound}
\ba{c}
O\bigl( \max\{ (\gammafnu)^{\frac{-1}{2 + \nu}}, 1\} 
\cdot \log \frac{\Hfnu D_0^{2 + \nu}}{\varepsilon}  \bigr).
\ea
\eeq
This is the better rate  than~\eqref{OGammaComplexity}.
However, the method requires to know the parameter $\nu$,
and the constant of uniform convexity. Thus, one
theoretical question remains open:
is it possible to construct \textit{universal}
second-order scheme, matching~\eqref{OAccBound}
in the uniformly convex case.

Looking at the definitions of $\Hfnu$ and $\sigmafnu$, we can see that,
for all $x, y \in \dom F, x \not= y$,
\begin{displaymath}
\ba{rcccl}
\sigmafnu & \; \leq \; &
\frac{\langle \nabla f(x) - \nabla f(y), x - y\rangle}
{\|x - y\|^{2 + \nu}}, \quad \frac{1}{\Hfnu}
& \; \leq \; &
\frac{\|x - y\|^{\nu}}{\| \nabla^2 f(x) - \nabla^2 f(y) \|},
\ea
\end{displaymath}
and
\begin{displaymath}
\ba{rcl}
\gammafnu & \; := \; & \frac{\sigmafnu}{\Hfnu}
 \; \leq \; 
\frac{\langle \nabla f(x) - \nabla f(y), x - y \rangle}
{\| \nabla^2 f(x) - \nabla^2 f(y)\| \cdot \|x - y\|^2}.
\ea
\end{displaymath}
The last fraction does not depend on any particular $\nu$.
So, for any twice-differentiable convex function, we can
define the following number:
\begin{displaymath}
\ba{rcl}
\gamma_f & \defeq & \inf\limits_{\substack{ x, y \in \dom F \\
		x \not= y }} \frac{\langle \nabla f(x) - \nabla f(y), x - y \rangle }
{\| \nabla^2 f(x) - \nabla^2 f(y) \| \cdot \|x - y\|^2}.
\ea
\end{displaymath}
If it is positive, then it could serve as an indicator of the
\textit{second-order non-degeneracy}, for which we
have a lower bound:
$
\gamma_f \geq  \gammafnu, \; \nu \in [0, 1].
$

\section{Conclusions}

In this work, we have introduced the second-order condition
number of a certain degree, which plays as the main complexity
factor for solving uniformly convex minimization problems
with H\"older-continuous Hessian of the objective by second-order 
optimization schemes.

We have proved that cubically regularized Newton method with
an adaptive estimation of the regularization parameter
achieves global linear rate of convergence on this class of functions.
The algorithm does not require to know any parameters of 
the problem class and 
automatically fits to the best possible degree of nondegeneracy. 

Using this technique, we have justified that 
global iteration complexity 
of Cubic Newton is always better
than corresponding one of the gradient method
for the standard class of strongly convex functions 
with uniformly bounded second derivative.

Data sharing not applicable to this article as no datasets were generated or analysed during the current study.

\begin{acknowledgements}
The research results of this paper were
obtained with support of ERC Advanced Grant 788368.
\end{acknowledgements}





                       





\end{document}